\documentclass[fleqn,a4paper,11pt]{article}
\usepackage{amsmath,amsfonts,calrsfs,amssymb,color,dsfont}
\usepackage{amsthm}

\newtheorem{Theorem}{Theorem}[section]

\newtheorem{Proposition}[Theorem]{Proposition}
\newtheorem{Lemma}[Theorem]{Lemma}
\newtheorem{Corollary}[Theorem]{Corollary}
\newtheorem{Remark}[Theorem]{Remark}

\newtheorem{Hypothesis}[Theorem]{Hypothesis}
\makeatletter
\@addtoreset{equation}{section}

\makeatother

\numberwithin{equation}{section}

\def\R{\mathbb R}

\def\E{\mathbb E}
\def\P{\mathbb P}
\def\ds{\displaystyle}

\def\
{\color{blue}}

\title{\bf Gradient estimates for SDEs  without monotonicity type conditions 
}
\author{Giuseppe Da Prato \thanks{Piazza dei Cavalieri 7, 56126 Pisa, Italy; g.daprato@sns.it}
\\\normalsize Scuola Normale Superiore \\\\
Enrico Priola \thanks{Via Carlo Alberto 10, 10123 Torino, Italy; enrico.priola@unito.it}
\\\normalsize Universit\`a di Torino \\ \normalsize Dipartimento di Matematica ``G. Peano''}
\date{ }

\begin{document}
\maketitle

\begin{abstract}
We prove gradient estimates for  transition Markov  semigroups $(P_t)$ associated to   
SDEs driven by  multiplicative Brownian noise  having possibly unbounded $C^1$-coefficients, without  requiring any monotonicity type condition.
 In particular,   first  derivatives of coefficients can grow polynomially and even exponentially.
We establish  pointwise  estimates  with  weights for  $D_x P_t\varphi$ 
 of the form
$$
{\sqrt{t}}  \, |D_x P_t \varphi (x) | \le 
c \, (1+ |x|^k) \, \| \varphi\|_{\infty}, \;\;\;  \text{$t \in (0,1]$, $\varphi \in C_b (\R^d)$, $x \in \R^d.$}
$$
To prove the result we use two main tools.
 First, we consider  a  Feynman--Kac semigroup 
with  potential $V$ related to the growth of the
coefficients   and of their derivatives    
for which we can use a Bismut-Elworthy-Li type formula.  Second,  we introduce  a new regular approximation for
the coefficients of the SDE.
At the end of the paper we  provide 
an   example of SDE with additive noise and  drift $b$ 
 having sublinear growth together with its  derivative 
 such  that   
 uniform estimates 
 for $D_x P_t \varphi$
 without weights 
 do not hold.
 \end{abstract}

\bigskip

\noindent {\bf 2010 Mathematics Subject Classification AMS}:  60H10,  60H30, 47D07.  \medskip

\medskip

\section{ Introduction and main
result}

We consider   the transition Markov semigroup $(P_t) $ associated to the SDE
 \begin{equation}
\label{e1.1}
\left\{\begin{array}{l}
dX(t)= b(X(t))dt+ \sigma (X(t))dW(t),\\
X(0)=x\in \R^d,
\end{array}\right. 
\end{equation}
where  $W= (W(t))$ is an $\R^d$--valued standard Brownian motion  defined   on a fixed stochastic basis $(\Omega, {\cal F}, ({\cal F}_t), \P)$ (see, for instance, \cite{Kr95} and \cite{Ma08})  and $b : \R^d \to \R^d$ and $\sigma: \R^d \to L(\R^d) = \R^d \otimes \R^d$ are $C^1-$functions.
  We write
$$
\sigma (X(t))dW(t) = \sum_{i=1}^d \sigma_i (X(t)) dW^i(t),
$$
where $W^i= (W^i(t))$ are independent real Brownian motions, $i=1, \ldots, d$.
We have  \begin{equation}
\label{e2.3a}
P_t\varphi(x)=\E[\varphi(X(t,x))],\quad t\ge 0,\; x\in {\R^d},\;\varphi\in C_b({\R^d}).
\end{equation} 
We will prove pointwise gradient estimates such as
\begin{equation} \label{222}
{\sqrt{t}}  \, |D_x P_t \varphi (x) | \le 
c \, w(x) \, \| \varphi\|_{\infty}, \;\;\;  \text{$t \in (0,1]$, $\varphi \in C_b (\R^d)$, $x \in \R^d,$}
\end{equation}
where $w(x)$ is a suitable weight related to the growth of 
$b$ and $\sigma$  and of their derivatives. 

There are several recent papers on uniform gradient estimates without weights  for non-degenerate
 diffusion semigroups which   are of the form
\begin{equation}\label{gr2} 
{\sqrt{t}}  \, |D_x P_t \varphi (x) | \le 
c \, \, \| \varphi\|_{\infty}, \;\;\;  \text{$t \in (0,1]$, $\varphi \in C_b (\R^d)$, $x \in \R^d.$}
\end{equation}
Such papers deal with 
the case of unbounded coefficients and even with possibly irregular coefficients.
We refer to \cite{El92},  \cite{Ce}, \cite{BeFo04},  \cite{PrWa06} and  \cite{PrPo13} and the references therein. Uniform gradient estimates related  to some non-linear parabolic equations are also considered in \cite{PrPo13}.
 We also mention that $D_xP_T \varphi(x)$ is of interest in  financial mathematics where it is called    Greek: it represents the rate of change  of the price of the derivative at time $T$ with respect to the initial prices; see Chapter 6 in \cite{N}. 

 Assuming regular unbounded coefficients and the existence of a Lyapunov function, a typical  monotonicity condition one imposes to obtain \eqref{gr2} is 
\begin{equation}\label{de22} 
{2}\langle Db(x) h,h \rangle  + \sum_{i=1}^d | D\sigma_i(x) h |^2 \le C |h|^2, \quad\;   x,h\in \R^d
\end{equation}
(we denote by   $|\cdot|$ and  $\langle\cdot,\cdot \rangle$   the norm and the inner product in  ${\R^d}$, $d \ge 1$). 
Recall that for the class of non-degenerate diffusions  with   $b$ and $\sigma$   globally bounded,   estimates   \eqref{gr2}  holds without any control  on the derivatives, see e.g. Section 5 in  \cite{St}.

In this paper we consider cases in which $b$ and $\sigma $ are regular, possibly unbounded but  
condition \eqref{de22} does not hold in general.  In such situations we can  prove  gradient estimates \eqref{222} with weights. Moreover, we  provide
a one-dimensional   example of SDE with additive noise and  drift $b$ 
with  sublinear growth 
 together with its  derivative 
 such  that 
 uniform estimates \eqref{gr2}
 do not hold (see Section 6).

Using the notation $a(x)  = \sigma(x) \sigma^*(x),$  where $\sigma^*(x)$ is the transposed matrix of $\sigma(x)$, we make the following assumptions. 
\begin{Hypothesis} \label{h1}  Let  $\gamma \in [1/2,1]$,
   $M_0 >0$ and $c_0 \ge 1$ and consider 
\begin{equation} \label{w55}
f(t) = M_0  e^{c_0 \int_1^t { s^{-\gamma} } \, ds}, \;\;\; t \ge 1. 
\end{equation}
(H1) There exists  $L >0$ such that, for any $x \in \R^d$, 
\begin{gather*}
\nonumber
{ 2 \langle b(x),x   \rangle + 
Tr(a(x) ) \, + \, \frac{ 8 c_0 \, \langle a(x) x,x   \rangle }{ (1+ |x|^2)^{\gamma}}
\le  L \,  (1+ |x|^2)^{\gamma}.} 
\end{gather*}
(H2) We have,    for any $x  \in \R^d$,  for any $h \in \R^d,$ with $|h| =1$,
\begin{equation}
\label{e1.3} {
 {2}|Db(x) h|  + \sum_{i=1}^d | D\sigma_i(x) h |^2 \le   \, f (1+|x|^{2}).
}
\end{equation} 
(H3)  $\sigma(x)$ is invertible, $x \in \R^d$. Moreover   there exists $\nu >0$ such that
\begin{equation}
\label{inv}
\langle a(x) h,h \rangle \ge \nu |h|^2,\;\;\; x , \, h \in \R^d.
\end{equation} 
\end{Hypothesis}
 The above  function $f$  allows  for non-standard growth of 
the derivatives of $b$ and $\sigma$ (see also Remarks \ref{cases} and  \ref{serv1}). For instance, when $\gamma =1$ we have $f(t) = M_0 t^{c_0} $ and  when $\gamma =1/2$ we  have $f(t)= M e^{2c_0 \sqrt{t}}$, $t \ge 1$.  However   there is a  balance between the growth of the coefficients 
 and the growth of its first derivatives. 

Under the previous assumptions we  prove gradient estimates \eqref{222} with weight 
$$
 w(x) = [f(1+ |x|^2)]^{2},\;\;\; x \in \R^d.
$$
 Choosing  $\gamma =1$, we can allow
 $\sigma(x)$ to grow at most linearly and moreover
$
 \langle b(x),x   \rangle 
$ $ \le  C \,  (1+ |x|^2),$ $x\in \R^d. $ In this case,  we can allow
 the following polynomial growth of the first derivatives:
$$
{2}|Db(x) h|  + \sum_{i=1}^d | D\sigma_i(x) h |^2 \le M_0  \, (1+|x|^{2})^{c_0}, \;\;  x \in \R^d, \;\;
h\in \R^d,\; |h|=1,
$$
 and  obtain ${\sqrt{t }}  \, |D_x P_t \varphi (x) | \le C \,   (1+|x|^{2})^{2c_0}
  \, \| \varphi\|_{\infty}$, for any $\varphi \in C_b(\R^d)$. 
  A simple  one-dimensional example  is 
$$ 
dX(t) = X(t) \sin(X^2(t)) dt+ dW(t),\;\; X(0)=x;
$$  in this case $f(t) = 2 t$  ($c_0 =1
$ and $\gamma=1$). 
 Other cases which we can consider are collected in Remark \ref{cases}.

   Existence and uniqueness  of a strong solution $X(\cdot,x)$ to \eqref{e1.1}
 is standard since (H1) implies 
the well-known non-explosion  condition
\begin{equation} \label{111} 
2 \langle b(x),x   \rangle  +  Tr(a(x) )     \le  L (1+ |x|^2),\quad  x\in \R^d, 
\end{equation}
 see e.g. the monographs  \cite{Kr95} and  \cite{Ma08}.   We will also use hypothesis   (H1)  to get
 $L^p$-estimates for  $f( |X(\cdot,x)|^2 + 1)$ (see  Proposition \ref{l2.1} and Corollary \ref{aa2}). 

  Proving  our   estimates for $DP_t\varphi  = D_xP_t\varphi$  requires some work  because, due to the  growth  of the derivatives of $b$ and $\sigma$,   we cannot exploit the classical   Bismut--Elworthy--Li formula, see \cite{El92}. Indeed under our assumptions we do not expect to have $L^2$-estimates for the derivative $D_x X(t,x)$ which appears in the  classical  Bismut--Elworthy--Li formula (cf. Lemma \ref{l1.2}).   Such $L^2-$estimates would lead to uniform gradient estimates which do not hold in general (cf. Section 6).

We  prove   weighted gradient estimates   inspired by   \cite{DaDe03}, \cite{DaDe07}  and  \cite{DaDe14},  introducing a suitable   potential $V$ related to  Hypothesis \ref{h1}:
\begin{equation} \label{gee}
V(x)=f(1+|x|^{2}),\;\;\; x \in \R^d,
\end{equation}
  and studying the corresponding Feynman--Kac semigroup
 \begin{equation}
\label{a1}
S_t\varphi(x)=\E \Big [\varphi(X(t,x))\,e^{-\int_0^tV(X(s,x))\,ds} \Big],\;\;\; t \ge 0, \; x \in \R^d.
\end{equation}
More precisely, the Feynman--Kac semigroup we consider  
 is a non-standard  regular approximation of $S_t$ (cf. Sections 3 and 4.2). 
 We shall first  prove gradient estimates for $ \langle DS_t\varphi(x),h   \rangle$   then we will return to  $ \langle DP_t\varphi(x),h   \rangle$, using   the identity
\begin{equation}
\label{a2}
P_t\varphi(x)=S_t\varphi(x)+\int_0^t S_{t-s}(V P_s\varphi)(x)\,ds
\end{equation}
which follows from the variation of constants formula,  
 see Section 4.3. Indeed  denoting by $\mathcal L$ and
$\mathcal K$ the 
generators of $P_t$ and $S_t$ respectively, we have
$$
\mathcal L=\mathcal K+V. 
$$ 
 In order to estimate
 $ \langle DS_t\varphi(x),h   \rangle$  
 we  deal with a probabilistic formula for the gradient of  $S_t \varphi$ which holds when $b$, $\sigma$ and $V$ are  Lipschitz $C^1$-functions (see  \cite{DaZa97}).  
  This is why we introduce new regular
approximations for $b$, $\sigma $ and $V$ in Section 3. Finally, we   use a local regularity result as in \cite{MPW}  and 
 a localization argument  to pass from Lipschitz estimates for the approximating SDEs with regular   coefficients to Lipschitz estimates  for  the general SDE satisfying our assumptions (cf. Section 5). Our main result is the following:
\begin{Theorem} \label{mai} Under Hypothesis \ref{h1}, there exists $c= c(c_0, \gamma,M_0, L,   d) >0$ such that for any  $\varphi \in C_b(\R^d)$,  we have
\begin {equation} \label{dop1}
{ \sqrt{t \nu} }  \, |D_x P_t \varphi (x) | \le 
c \, \, { V^{2}(x)}\| \varphi\|_{\infty}, \;\;\;  \text{$t \in (0,1]$,  $x \in \R^d.$}
\end{equation}
\end{Theorem}

 
Below, we make some comments. 
\smallskip 

\begin{Remark} \label{cases} {\em (i) Choosing $\gamma=\tfrac12$, we can consider the case when 
$a(x)$  is globally bounded and there exists $c>0$ such that $\langle b(x),x \rangle$ $\le c$, for any $x \in \R^d$.
 In this case $f(t) =M e^{2 c_0 \sqrt{t}}$
and 
 (H2) becomes   
$$
{2}|Db(x) h |  + \sum_{i=1}^d | D\sigma_i(x) h |^2 \le C  \, e^{2c_0\, \sqrt{1+|x|^{2}}}, \quad\;  
x \in \R^d,
$$ 
$\,h\in \R^d,$ $|h|=1$;
 our gradient estimates are
\begin{gather*} 
{\sqrt{t \nu}}  \, |D_x P_t \varphi (x) | \le c \, e^{4c_0 \, \sqrt{1+|x|^{2}}}
\,  \, \| \varphi\|_{\infty}, \;\;\;  \text{$t \in (0,1]$, $\varphi \in C_b (\R^d)$, $x \in \R^d.$}
\end{gather*}
  }
\end{Remark}

\begin{Remark} \label{serv1} {\em 
 We stress that  $V(x) = f(1+ |x|^2)$ is a Lyapunov function for the Kolmogorov operator ${\cal L}$  associated to our SDE. More precisely, by (H1) even $V^4(x)$ is a Lyapunov function. This fact will be  used in Proposition \ref{l2.1} and Corollary \ref{aa2}. In our approach we will also use that (cf. \eqref{f44})
$$
|V'(x)| \le C |V(x)|, \;\; 
x \in \R^d.
$$ 
}
\end{Remark}

   We end the section with some notations. We denote by  $|\cdot|$ and  $\langle\cdot,\cdot \rangle$  or $\cdot$ the norm and the inner product in  ${\R^d},$ $d\ge 1$.  We indicate
by  $L({\R^d})$   the space of all $d \times d$-real matrices. If $A \in L({\R^d})$ then $Tr(A)$ indicates the trace of $A$. We use the Hilbert-Schmidt norm $\| \cdot \|$ on $L({\R^d}).$
Moreover, $C_b({\R^d})$   is  the space of all real continuous   and bounded mappings   $\varphi\colon {\R^d}\to \R$ endowed with the sup norm 
  $$\|\varphi \|_{\infty}=\sup_{x\in  {\R^d}}\,|\varphi(x)|;
  $$
  whereas $C^k_b({\R^d})$, $k \ge 1,$ is  the space of all real functions which are continuous and bounded together with their derivatives 
  of order less or equal than $k$. 
     Finally,   $ B_b({\R^d})$ is the space of  all  real,  bounded   and Borel mappings on ${\R^d}$.\smallskip

 We denote by $C$ or $c$   generic positive constants that might change from line to line and that may depend on  $c_0$, $\gamma,$ $M_0,$ $L$ and   $d$ if it is not specified.

\section{Preliminary  estimates for the   solution $X(t,x)$  }

We first present  estimates for  $\E[V(|X(t,x))]$. Recall that in Theorem 4.1 of \cite{Ma08} estimates for $\E(|X(t,x)|^{2m})$ are given.  
 Here  we are considering more general estimates (for instance,   if $\gamma = 1/2$ we are establishing exponential type estimates because in this case $f(s) \sim e^{2 c_0 \sqrt{s}}$). Estimates as in \cite{Ma08} are obtained by choosing $\gamma =1$.

  Note that condition \eqref{e1.216} is weaker than (H1) which has $8 c_0$ instead of $2 c_0$.
\begin{Proposition}
\label{l2.1}  Let $\gamma \in [1/2,1]$ and $f$ as in \eqref{w55}. Moreover suppose that
\begin{equation}
\label{e1.216}
2 \langle b(x),x   \rangle  +  Tr(a(x) ) \,  + \,   
\tfrac{2c_0 }{(1+ |x|^2)^{\gamma}}
\langle a(x)x,x   \rangle \le {L}  (1+ |x|^2)^{\gamma},\quad x\in \R^d,
\end{equation}
for some $L>0$. Then the solution $X(t,x)$  verifies
\begin{equation}
\label{e2.1}
\E[f \big( |X(t,x)|^{2} + 1 \big)]\le   e^{   { c_0 L}\, t} f(|x|^{2} +1),
\quad \;x\in {\R^d},\;t\ge 0.
\end{equation}
\end{Proposition}
\begin{proof} Recalling \eqref{gee}  
 in the sequel we set  
\begin{equation} \label{cit1}
\text{ $g(t) = t^{\gamma}$ so that} \;\; f(t) =  M_0  e^{c_0 \int_1^t \frac{1}{g(s)} \, ds}, \;\;\; t \ge 1. 
\end{equation}
Let us fix $t>0$.
 We  apply It\^o's formula setting $ X(t,x) =X(t)$
\begin{gather*}
V(X(t)) = V(x) +  \int_0^t 2 f'(1 + |X(s)|^2)^{} X(s) \cdot b(X(s)) ds 
\\
 \nonumber
+ 2 \int_0^t f'(1 + |X(s)|^2) \; X(s) \cdot \sigma(X(s)) dW(s) 
\\ \nonumber
+   2\int_0^t f''(1 + |X(s)|^2)^{} \,  |\sigma^*(X(s)) X(s)|^2 ds
\\
+  \, \int_0^t f'(1 + |X(s)|^2)^{} \| \sigma(X(s)) \|^2 ds. 
\end{gather*}
Using that 
$$
f'(s) = \tfrac{c_0 f(s)}{g(s)},\;\;\; 
f''(s) \le f(s) \tfrac{c_0^2 }{g^2(s)} = f'(s) \tfrac{c_0}{g(s)},\;\; s \ge 1,
$$
we obtain 
\begin{gather*}
V(X(t)) \le  V(x) +  2\int_0^t f'(1 + |X(s)|^2)^{} \; X(s) \cdot \sigma(X(s)) dW(s)
\\
 +  \Big [ \int_0^t f'(1 + |X(s)|^2)
  \Big\{ 2 X(s) \cdot b(X(s)) 
+   
 \tfrac{2 c_0 }{g(1 + |X(s)|^2)^{}} \,  |\sigma^*(X(s)) X(s)|^2    \\ 
 + \, \| \sigma(X(s)) \|_{}^2 \Big \}ds \Big] 
\\
\le V(x) + 2\int_0^t f'(1 + |X(s)|^2)^{} \; X(s) \cdot \sigma(X(s)) dW(s)
\\
 + \,  {L} \int_0^t f'(1 + |X(s)|^2) g(1 + |X(s)|^2) ds
\\
\le V(x) + 2\int_0^t f'(1 + |X(s)|^2)^{} \; X(s) \cdot \sigma(X(s)) dW(s)+ c_0 {L}  \int_0^t V(X(s))   ds. 
\end{gather*}
Using the  stopping times 
$
\tau_n = \tau_n(x) $ $ = \inf \{t \ge 0 \; :\; |X(t)| \ge n \}
$
and taking the expectation  we get 
$$
\E[ V (X (t \wedge \tau_n )] \le V(x)  
 +  c_0 {L}  \int_0^t \E [ V (X(s \wedge \tau_n ) ) ]    ds. 
$$ 
By the Gronwall lemma we get an estimate for  $\E [ V (X(t \wedge \tau_n ) )]$; letting $n \to \infty$  (note that  $\tau_n \uparrow \infty$ because condition \eqref{111} holds) we 
get the assertion.
\end{proof}

\begin{Corollary}
\label{aa2}   Let $\gamma \in [1/2,1]$ 
and $f$ as in \eqref{w55}.
 Moreover suppose that
\begin{equation}
\label{e1.21}
2 \langle b(x),x   \rangle  +  Tr(a(x) ) \,  + \,   
\tfrac{8 c_0}{(1+ |x|^2)^{\gamma}}
\langle a(x)x,x   \rangle \le {L}  (1+ |x|^2)^{\gamma},\quad x\in \R^d.
\end{equation}
 Then the solution $X(t,x)$  verifies
\begin{equation}
\label{e2.1a} \E[V^4(X(t,x))]=
\E[f^4 \big (|X(t,x)|^{2} + 1 \big)]\le   e^{  4 c_0 {L} t} \, V^4(x),
\quad x\in {\R^d},\;t\ge 0.
\end{equation}
\end{Corollary}
\begin{proof}  Taking into account that \eqref{e1.21} is like 
 \eqref{cit1} with $c_0$ replaced by $4 c_0$, we define $\tilde f(t) $ $= M_0^4  e^{4 c_0 \int_1^t \tfrac{1}{g(s)} ds}, $ $ t \ge 1,$ and note that $\tilde f (t) = f^4(t)$. By  the previous proposition we obtain easily the result.
\end{proof}

\begin{Remark}{\em   The previous integral estimates \eqref{e2.1} and \eqref{e2.1a}  hold more generally  with a function $f : [1, \infty) \to \R $ as in  \eqref{cit1} and a corresponding      
$C^1$-function  $g: [1, \infty) \to \R$
such that 
$
1 \le g(s) \le s, $ for $
s \ge 1. $
\\
In such  case the assumptions must be changed replacing $(1+ |x|^2)^{\gamma}$ with $g(1+ |x|^2)$. For instance, assumption 
 \eqref{e1.21} becomes
 \begin{equation} \label{r666}
2 \langle b(x),x   \rangle  +  Tr(a(x) ) \,  + \,   
\tfrac{8c_0 }{g(1+ |x|^2)}
\langle a(x)x,x   \rangle \le {L}  g(1+ |x|^2),\quad x\in \R^d.
\end{equation}
One  could consider 
the more general condition \eqref{r666} instead of (H1) but then  
it is not clear how
to obtain the results 
of Section 3 with $g(t)$ instead of $t^{\gamma}$ (see in particular \eqref{ftt}). }
\end{Remark}

\section{ Regular approximations with bounded derivatives for $b$, $\sigma$ and $V$}

 Here we introduce Lipschitz $C^1$-approximations   of $b$, $\sigma$ and $V= f(1 + |\cdot|^2)$ which  satisfy hypotheses (H1), (H2) and (H3), possibly replacing $M_0$ and  $L$ with $cM_0$ and  $cL$ for some $c>0$ independent of $n$. 
Note that usual approximations for operators with unbounded coefficients do not work  (see, in particular,   \cite{L} and \cite{Ce}). 

Let $\eta \in C_0^{\infty}(\R^d)$ be a function such that $0 \le \eta(x) \le 1$, $x \in \R^d$, and $\eta(x)=1$ for $|x| \le 1$. Moreover, $\eta(x) =0 $ for $|x|\ge 2.$
Let us define the  $C^{\infty}$-mappings $\Phi_n : \R^d \to \R^d$,
 \begin{equation} \label{w77}
\Phi_n(y) = \frac{y}{\sqrt{ 1+ \big(1- \eta(\tfrac{y}{n}) \big) \, \tfrac{|y|^2}{n^2} }},\;\;\; y \in \R^d,\;\; n \ge 1.
\end{equation}
Note that $\Phi_n(y) =y$ if $|y| \le n$. Moreover, $|\Phi_n(y)| \le  |y|$, if $n \le |y| \le 2n$ and  $|\Phi_n(y)| \le  n$ if $|y| > 2n.$ Hence,
 for any $y \in \R^d$, we get
\begin{equation} \label{244}
 |\Phi_n(y)| \le  |y| \wedge 2n,\;\;\; n \ge 1.
\end{equation}
In addition, for any $y,h \in \R^d,$ when $n=1$,
$$
D\Phi_1(y)[h] = \tfrac{h}{\sqrt{ 1+ \big(1- \eta(y) \big) \, |y|^2 }}
$$$$
 - \tfrac{ y}{2}  \big (1+ \big(1- \eta(y) \big) \, |y|^2 \big)^{-3/2} \big [ - \langle D\eta(y) , h \rangle |y|^2  +  2\big(1- \eta(y) \big) \langle y , h \rangle\big].
$$
Hence, considering as before the cases $|y| \le 1$, $1 < |y| \le 2$ and $|y|>2$, we find that there exists $c = c(\| D\eta\|_{\infty})>0$ such that 
$$
\sup_{y \in \R^d} | D\Phi_1(y)[h] | \le c |h|, \;\;\; h \in \R^d.
$$ 
Writing $\Phi_n(y) = n \Phi_1(\tfrac{y}{n})$ we find that, for any $n \ge 1$, 
\begin{equation} \label{2ee1}
\sup_{y \in \R^d} | D\Phi_n(y)[h] | \le c \, |h|, \;\;\; h \in \R^d.
\end{equation}
  Recalling  that  $\| \Phi_n \|_{\infty} \le 2n$, we  consider the globally Lipschitz and bounded $C^1$-coefficients
\begin{equation} \label{bb1}
 b_n(x) = b(\Phi_n(x)),\;\;\; \sigma_n(x)= \sigma(\Phi_n(x)), \;\; x \in \R^d, \; n \ge 1.
\end{equation}
Note that, for any $n \ge 1$,
\begin{equation} \label{2ee2}
 b(x)= b_n(x),\;\; \sigma_n(x) =\sigma(x),
\; |x|\le n,\,\,\, 
\;\; x \in \R^d.
\end{equation}
Each $\sigma_n $ satisfies (H3) with the same $\nu$. Concerning (H1)
 we clearly have for $|x| \le n$:
$$
2 \langle b_n(x),x   \rangle + \;\;  Tr(a_n(x) ) \,  + \,   
  \, \tfrac{ 8 c_0 \, \langle a_n(x) x,x   \rangle }{ (1+ |x|^2)^\gamma}
 \le  L \,  (1+ |x|^2)^\gamma,
$$
where  $a_n(x) = \sigma_n(x) \sigma_n^*(x).$  Let us treat now the case
$|x|>n$. 

We can only consider the  case when $\langle b_n(x),x   \rangle >0.$ We find
\begin{gather*}
2 \langle b_n(x),x   \rangle = 2 \langle b(\Phi_n(x)),  \Phi_n(x)  \rangle 
{\sqrt{ 1+ \big(1- \eta(\tfrac{x}{n}) \big) \, \tfrac{|x|^2}{n^2} }}
\\ \le 
2 \langle b(\Phi_n(x)),  \Phi_n(x)  \rangle
\Big (1 + \tfrac{ |x|}{n} \Big),
\end{gather*}
and, similarly, 
\begin{gather*}
Tr(a_n(x) ) + \tfrac{ 8 c_0 \, \langle a_n(x) x,x\rangle} {(1+ |x|^2)^\gamma }  
\\ = 
 Tr(a_n(x)) + \, { 8 c_0 \, \langle a (\Phi_n(x)) \Phi_n(x) , \Phi_n(x)   \rangle }\,   \Big ({ 1+ \big(1- \eta(\tfrac{x}{n}) \big) \, \tfrac{|x|^2}{n^2} }\Big) \tfrac{1}{(1+ |x|^2)^{\gamma}}
\\
 \le  Tr(a(\Phi_n(x)) + 8 c_0
\langle a (\Phi_n(x)) \Phi_n(x), \Phi_n(x)\rangle \, (1+ \tfrac{ |x|^2}{n^2} ) \tfrac{1}{(1+ |x|^2)^{\gamma}}
\\
\le   Tr(a(\Phi_n(x))) + 8 c_0
\langle a (\Phi_n(x)) \Phi_n(x), \Phi_n(x)   \rangle \, (1+  \tfrac{ |x|^2}{n^2}) \tfrac{1}{|x|^{2\gamma}}
\\
\le (\tfrac{ |x|^{2(1-\gamma)} }{n^{2(1-\gamma)}} + 1)\, Tr(a(\Phi_n(x))) + 8 c_0
\langle a (\Phi_n(x)) \Phi_n(x), \Phi_n(x)   \rangle \, (\tfrac{ |x|^{2(1-\gamma)} }{n^{2(1-\gamma)}} +1).
\end{gather*}
Hence, when  
$|x| > n$  we have, since $\gamma \in [1/2,1]$, 
\begin{gather*}
2 \langle b_n(x),x   \rangle + 
Tr(a_n(x) ) + \tfrac{ 8 c_0 \, \langle a_n(x) x,x\rangle} {
(1+ |x|^2)^{\gamma} } 
\\
\le  c L\,  (1+ n^{2\gamma})\,\tfrac{ |x|}{n} +  c L \,  (1+ n^{2\gamma})\, \tfrac{ |x|^{2(1-\gamma)} }{n^{2(1-\gamma)}} 
\le  
2c L\,  ({ |x|^{2 - 2 \gamma}} + n^{4\gamma -2} |x|^{2 - 2 \gamma})
\\ 
\le  2c L\,  ( |x|^{2 - 2 \gamma} + |x|^{2\gamma }) \le 2 c L \, (1 + |x|^{2})^\gamma ,
\end{gather*}
since $|x|^{2(1-\gamma)} \le |x|^{2 \gamma}$ for $|x| \ge 1$, where $c>0$ is independent of $n$.

Hence coefficients $b_n$ and $\sigma_n$ satisfy (H1) with $L$ replaced by $C L$, 
 for some constant $C>0$ independent of $n$, i.e., we have:
 \begin {equation} \label{ftt}
2 \langle b_n(x),x   \rangle  +  Tr(a_n(x) ) \,  + \,   
  \, \tfrac{ 8 c_0 \, \langle a_n(x) x,x   \rangle }{(1+ |x|^2)^\gamma}
 \le  C L \,  (1+ |x|^2)^\gamma.
\end{equation}
Let us consider  (H2). If $|x|<n$ we get:
 $Db_n(x) = Db(x),$ $D(\sigma_n)_i(x) = D\sigma_i(x).$  Let now  $|x| \ge n$. We will use that if $y, k \in \R^d$, $k \not =0$, then
$$
{2}|Db(y) k|  + \sum_{i=1}^d | D\sigma_i(y) k |^2 
=   
{2}|Db(y) \tfrac{k}{|k|}| |k|  + \sum_{i=1}^d | D\sigma_i(y) \tfrac{k}{|k|} |^2 |k|^2
\le   \, V(y) (1+|k|^{2}).
$$
We find, for  $|x| \ge n$ and  $|h| =1$, 
\begin{gather}
\nonumber \sum_{i=1}^d |D(\sigma_n)_i(x)h|^2 +|  Db_n(x)h|
\\
\label{ett}
 =   |Db \big( \Phi_n(x) \big) D\Phi_n(x) h| 
+ \sum_{i=1}^d |D\sigma_i \big (\Phi_n(x)\big) D\Phi_n(x)h|^2
\le c_1  \, V(\Phi_n(x)),
\end{gather}
with $c_1>0$ independent of $n$. Let us define:
\begin {equation} \label{ettt}
V_n(x) = c_1 V(\Phi_n(x)),
\;\; x \in \R^d,\; n \ge 1.
 \end{equation}
 For any $x \in \R^d,$  $n \ge 1$, for any $h \in \R^d$ with $|h|=1$, we have
\begin{equation}\label{vv1} 
\sum_{i=1}^d |D(\sigma_n)_i(x)h|^2 +| Db_n(x)h| \le  V_n(x). 
\end{equation}
Moreover,   by \eqref{244} we get:
\begin{equation} \label{222a}
 V_n(x) = c_1\, V(\Phi_n(x)) =  c_1 f(1+ |\Phi_n(x)|^2) \le  c_1 V(x), \;\; x \in \R^d, \; n \ge 1.
\end{equation}
We introduce the approximating  SDEs:
\begin{equation}\label{app1} 
 \left\{\begin{array}{l}
dX_n(t)= b_n(X_n(t))dt+ \sigma_n (X_n(t))dW(t),\\
X_n(0)=x\in \R^d.
\end{array}\right. 
\end{equation}
Arguing as in Corollary \ref{aa2} and using \eqref{ftt} and \eqref{222a} 
  we obtain
\begin{Proposition} \label{sti1}
The solution $X_n(t,x)$ of \eqref{app1}  verifies, for any $x \in \R^d$,
\begin{equation}
\label{ww1} 
\E[V^4_n(X_n(t,x))]
\le  c_1^4 \,
\E[V^4(X_n(t,x))]
\le C e^{  c t} \, V^4(x),\;\; t \ge 0.
\end{equation}
\end{Proposition}
Finally,  since $\gamma \in [1/2,1],$ we get the following useful estimates, for any $x \in \R^d,$ 
\begin{equation}\label{f44}
|{V}'(x)| = 2 f'(1 + |x|^2) |x|=
2 c_0 \tfrac{f(1 + |x|^{2})}{(1 + |x|^2)^\gamma } |x|
\le 
 2 c_0 \tfrac{{V}(x)}{\sqrt{ 1 + |x|^2} }| x|
 \le 2c_0 V(x)
\end{equation}
and 
\begin{equation}\label{234} 
|V'_n(x)| \le C V^{}(x),\;   \;\; n \ge 1,
\end{equation} 
with $C$ independent of $n$. 

\section {Gradient estimates for SDEs with  $b_n$ and $\sigma_n$ }

  In this section  we prove the following crucial gradient type estimates for the  transition semigroup $(P_t^n)$ associated to the SDE \eqref{app1}.

\begin{Lemma} \label{fin1} There exists $c= c(c_0, M_0, L, \gamma,  d) >0$ such that, for any  $\varphi \in C_b(\R^d)$, $x, h \in \R^d,$  we have
\begin {equation} \label{dop1a}
{\sqrt{t \nu}}  \, |P_t^n \varphi(x + h) - P_t^n \varphi(x)| \le 
c \,|h| \, V^{2}(x)\| \varphi\|_{\infty}, \;\;\;  \text{$t \in (0,1]$,} 
\end{equation}
 where $P_t^n \varphi(x) = \E [\varphi (X_n(t,x))]$.
\end{Lemma}

\subsection{Estimates for the derivative process $D_x X_{n}(t,x)$ }

We fix $n \ge 1$. Here  we  give an estimate for  the derivative $D_x X_{n}(t,x)h,$ which we denote by $\eta^h_{n}(t,x),$ $h\in {\R^d}$ ($X_{n}(t,x)$ is the solution to \eqref{app1}). We also write
\begin{equation}
\label{ci1}
b'_{n}(x) = Db_{n}(x),\;\;\; (\sigma_{n})_i'(x) = D (\sigma_{n})_i(x),\;\; i =1, \ldots, d,\;\; x \in \R^d.
\end{equation} 
 It is   well  known that  $\eta^h_n(t,x)$ is a solution to the random equation
\begin{equation}
\label{e1.8}
\left\{\begin{array}{l}
\ds \tfrac{d}{dt}\,\eta^h_{n}(t,x)=b'_{n}(X_{n}(t,x))\eta^h_n(t,x) dt + \sum_{i=1}^d (\sigma_{n})_i'(X_{n}(t,x)) \eta^h_{n}(t,x) dW^i(t),
\\
\eta^h_{n}(0,x)=h
\end{array}\right. 
\end{equation}

 \begin{Lemma}
\label{l1.2} Using $V_n(x)$ defined in \eqref{ettt}
 the  following estimate holds:
\begin{equation}
\label{e1.9}
\E \Big [ e^{- \int_0^t V_{n}(X_{n}(s,x)) ds} \, |\eta^h_{n}(t,x)|^2 \Big] \le  \,|h|^2,\quad t\ge 0,\;x,h\in {\R^d}.
\end{equation}
\end{Lemma}
\begin{proof}   In the proof we write
$
X_{n}(t,x) = X_{n}(t),$ $\eta^h_{n}(t,x) = \eta_{n}(t)
$
and introduce the process 
$$
Z_{n}(t) = e^{- \int_0^t V_{n}(X_{n}(s)) ds} \, |\eta_{n} (t)|^2.
$$
 Since 
\begin{gather*}
d |\eta_{n}(t)|^2 = {2\langle \eta_{n}(t), d \eta_{n}(t)\rangle} = 2 \langle b'_{n}(X_{n}(t)) \eta_{n}(t), \eta_{n}(t) \rangle dt
\\ + 2 \sum_{i=1}^d   \langle \eta_{n}(t), (\sigma_{n})_i'(X_{n}(t)) \eta_{n}(t) \rangle dW^i_t
+ \sum_{i=1}^d   |(\sigma_{n})_i'(X_{n}(t)) \eta_{n}(t)|^2 dt
\end{gather*}
we get
$$
\begin{aligned}
dZ_{n} (t) =  e^{- \int_0^t V_{n}(X(s)) ds} \Big [  \big (2 \langle b'_{n}(X_{n}(t)) 
  \eta_{n}(t), \eta_{n}(t) \rangle 
\\ + \sum_{i=1}^d   |\sigma_i'(X_{n}(t)) \eta_{n}(t)|^2  - V_{n}(X_{n}(t) |\eta_{n}(t)|^2 \big)dt
 + 2 \sum_{i=1}^d   \langle \eta_{n}(t), (\sigma_{n})_i'(X_{n}(t)) \eta_{n}(t) \rangle dW^i_t
\Big].
\end{aligned}
$$
Using    \eqref{vv1}  we get 
\begin{gather*}
dZ_{n} (t) \le  e^{- \int_0^t V_{n}(X(s)) ds} [ V_{n}(X_{n}(t)) - V_{n}(X_{n}(t))] \,  |\eta_{n}(t)|^2  dt
\\ + 2 \sum_{i=1}^d   \langle \eta_{n}(t), (\sigma_{n})_i'(X_{n}(t)) \eta_{n}(t) \rangle dW^i(t)
\Big]
\end{gather*}
and so, $\P$-a.s.,
$$
Z_{n}(t) \le  |h|^2 + 2 \sum_{i=1}^d  \int_0^t \langle \eta_{n}(s),(\sigma_{n})_i'(X_{n}(s)) \eta_{n}(s) \rangle dW^i(s).
$$
We find
$$
\E [Z_{n}(t)] \le |h|^2,\;\;\; t \ge 0,
$$
and the assertion holds.
\end{proof}

\subsection{Gradient  estimates for   the Feynman--Kac semigroup $S_t^n$}

As we said in the introduction, 
 we cannot  estimate  $D_x P_t^n\varphi$ for $\varphi\in C_b({\R^d})$  (uniformly in $n$) directly  using the   Bismut--Elworthy--Li  formula (see \cite{El92});  it would be necessary to estimate $|\eta_x^h(t,x)|$ whereas 
 we are only able  
 to show \eqref{e1.9}. 
  For this reason,  we consider the    potential
  $
  {V_{n}}(x),\; \, x\in {\R^d},
  $ given in \eqref{ettt}
  and the Feynman--Kac semigroup $S_t^{n}$ given by
$$
S_t^{n}\varphi(x)=\E[\varphi(X_{n}(t,x))\,e^{-\int_0^t{V_{n}}(X_{n}(s,x))\,ds}],\;\; \varphi \in B_b(\R^d), \; t \ge 0,\; x \in \R^d.
$$
 We recall that the Bismut--Elworthy--Li formula generalises to $S_t^{n}$, thanks to the regularity of $b_n$, $\sigma_n$ and $V_n$ (see \cite{DaZa97}). 
  In fact for all $\varphi\in C_b({\R^d})$, setting
 \begin{equation}
 \label{be1}
 \beta_{n}(t) =\beta_{n}(t,x) =\int_0^t {V_{n}}(X_{n}(s,x))ds, 
\;\; t \ge 0,
\end{equation}
 we have that $S_t^{n}\varphi$ is differentiable on $\R^d$, $t>0$, and 
 the following identity holds 
(we write $D{V_{n}}(x)={V'_{n}}(x)$, $x \in \R^d$)
\begin{equation}
\label{e1.11}
\begin{array}{l} \ds 
  \langle DS_t^{n}\varphi(x),h   \rangle =\tfrac1t\,\E \left[ \varphi(X_{n}(t,x))\,e^{-{\beta_{n}}(t)}\,\int_0^t  \langle \sigma^{-1}_{n}(X_{n}(s,x)) \, \eta^h_{n}(s,x),dW(s)   \rangle \right]\\
\\
\ds-\E \left[ \varphi(X_{n}(t,x))\,e^{-{\beta_{n}}(t)}\int_0^t\left(1-\tfrac{s}{t}   \right)\,\langle    {V_{n}}'(X_{n}(s,x)),\eta^h_{n}(s,x)  \rangle \,ds \right]\\
\\
=: I_1(\varphi,x,h,t,n)+I_2(\varphi,x,h,t,n)=I_1+I_2,\;\; t>0, \, \, x,h \in \R^d.
\end{array} 
\end{equation}
  \begin{Lemma}
\label{l1.3}
Let  $\varphi\in  C_b({\R^d})$, $t \in (0,1]$,  $x\in {\R^d}$. Then 
we have with 
 $C = C(c_0, \gamma,M_0, L,   d) >0$
\begin{equation}
\label{e13}
 |DS_t^{n}\varphi(x)| \le  \| \varphi\|_{\infty} 
\tfrac{C}{\sqrt {\nu \,t}}   
{V}(x)^{}. 
 \end{equation}
\end{Lemma}
\begin{proof}   In the proof we write
$
X_{n}(t,x) = X_{n}(t),$ $\eta^h_{n}(t,x) = \eta_{n}(t)
$ and consider the process ${\beta_{n}}(t)$ given in \eqref{be1}.
 We first consider  $ I_1.$ We have
\begin{equation}
\label{e14e}
\begin{array}{l}
|I_1|^2 \le  \ds \tfrac1{t^2}\left[\E\left( \varphi^2(X_{n}(t)\right)\right] \,  \E\left(e^{-2 {\beta_{n}}(t)}\left|\int_0^t \langle   \sigma^{-1}_{n}(X_{n}(s)) \eta_{n}(s),dW(s)\rangle\right|^2    \right)  \\
\\
\hspace{6mm}\le \ds \tfrac1{t^2}\left[\E\left( \varphi^2(X_{n}(t)\right)\right] \E\left(|z_{n}(t)|^2  \right),
\end{array} 
\end{equation}
where
\begin{equation}
\label{e15e}
z_{n}(t)=e^{- {\beta_{n}}(t)}\int_0^t \langle  \sigma^{-1}_{n}(X_{n}(s)) \eta_{n}(s),dW(s)\rangle,\quad t\ge 0.
\end{equation}
We now apply It\^o's formula: 
$$
\begin{array}{lll}
dz_{n}(t)&=&\ds-{\beta_{n}}'(t)e^{- {\beta_{n}}(t)}\int_0^t \langle  
\sigma^{-1}_{n}(X_{n}(s))
\eta_{n}(s),dW(s)\rangle\,dt \\ \\ &+& e^{- {\beta_{n}}(t)}  \langle \sigma^{-1}_{n}(X_{n}(t))  \eta_{n}(t),dW(t)\rangle\\ \\
&=&\ds- {\beta_{n}}'(t)z_{n}(t)dt+ e^{- {\beta_{n}}(t)}  \langle \sigma^{-1}_{n}(X_{n}(t)) \eta_{n}(t),dW(t)\rangle;
\end{array}
$$ 
we find
$$
\begin{array}{lll}
\ds d (z_{n}(t)^2) &=& 2 z_{n}(t) \big(-{\beta_{n}}'(t)z_{n}(t)dt+e^{- {\beta_{n}}(t)}  \langle   \sigma^{-1}_{n}(X_{n}(t)) \eta_{n}(t),dW(t)\rangle \big)\\
\\
&&\ds \, + \, e^{-2{\beta_{n}}(t)}| \, \sigma^{-1}_{n}(X_{n}(t)) \eta_{n}(t)|^2 \, dt.
\end{array} 
$$
Integrating, we find 
\begin{equation}
\label{e16f}
\begin{array}{lll}
 |z_{n}(t)|^2 &=&\ds-2\int_0^t|z(s)|^{2} {\beta_{n}}'(s)\,ds
+2\int_0^t z(s)  e^{- {\beta_{n}}(s)}  \langle  \sigma^{-1}_{n}(X_{n}(s)) \eta (s),dW(s)\rangle\\
\\
&&\ds+ \int_0^te^{-2{\beta_{n}}(s)} | \sigma^{-1}_{n}(X_{n}(s)) \eta_{n}(s)|^2ds.
\end{array} 
\end{equation}
Neglecting the negative  term in \eqref{e16f}  and taking expectation, we find by Lemma 
\ref{l1.2} (recall that  (H3) implies  
 $| \sigma^{-1}_{n}(x) h|^2 \le \tfrac{|h|^2}{\nu},$ $x ,$ $ h \in \R^d).$  
\begin{gather*}
\E [ |z_{n}(t)|^2 ] \le \E \Big[  \int_0^te^{-2{\beta_{n}}(s)} | \sigma^{-1}_{n}(X_{n}(s)) \eta_{n}(s)|^2ds\Big]
\\
=  \int_0^t \E \big [e^{-2 \int_0^s {V_{n}}(X_{n}(r)) dr} | \sigma^{-1}_{n}(X_{n}(s)) \eta_{n}(s)|^2\big] ds \le  \tfrac{c |h|^2 \, t}{\nu}, \;\;\; t \ge 0.
\end{gather*}
Coming back to \eqref{e14e} we obtain 
 \begin{equation}
\label{ee}
|I_1 (\varphi,x,h,t,n)|^2 \le    \ds \tfrac{c |h|^2}{\nu \, t}\left[\E\left( \varphi^2(X_{n}(t,x)\right)\right].
\end{equation}
Now we treat  $I_2 = I_2 (\varphi,x,h,t,n)$. 
$$
I_2 = -\E \left[ \varphi(X_{n}(t))\,e^{-{\beta_{n}}(t)}\int_0^t\left(1-\tfrac{s}{t}   \right)\,\langle    {V'_{n}}(X_{n}(s)),\eta_{n}(s)  \rangle \,ds \right ].
$$
Using \eqref{234}, we know that $|{V'_{n}}(X_{n}(s))| \le C {V}(X_{n}(s))$; 
we deduce
 \begin{equation}
\label{e29}
|I_2|
\le c \, \| \varphi\|_{\infty} \,   \Lambda(t),
\end{equation}
where 
\begin{equation}
\label{e30a}
\begin{array}{l}
\ds\Lambda(t)=\E\left[e^{-{\beta_{n}}(t)} \int_0^t {V}^{}(X_{n}(s))  \,|\eta_{n}(s)|\,ds \right] \\
\\
\ds\le  \int_0^t \E \left[  {V}^{}(X_{n}(s))  \, e^{-{\beta_{n}}(s)} \, |\eta_{n}(s)|\right ]\,ds   \\
\\
\ds\le  \left (\int_0^t \E \left[{V}^{2}(X_{n}(r))\right] dr\right)^{1/2} 
 \left ( \int_0^t \E \left [ e^{-2\beta_{n}(s)}\,|\eta_{n}(s)|^2 \right] ds \right)^{1/2}.
 \end{array}
\end{equation} 
Now we use  \eqref{ww1} and  \eqref{e1.9}  to get 
\begin{equation}
\label{e31}
|I_2|
\le  C \, t^{1/2} \, |h| \| \varphi\|_{\infty} 
{V}^{}(x), \;\; t \in (0,1]. 
 \end{equation}
 Finally, by \eqref{ee}  and \eqref{e31}, we get the assertion.
\end{proof}

\subsection{Proof of  Lemma \ref{fin1}}

We first remark that,
for any $\varphi \in C_b (\R^d)$, $v(t,x) = S_t^{n} \varphi (x)$ solves
$$
\begin{cases}
\partial_t v(t,x) = {\cal L}_{n} v(t,x) - {V_{n}}(x) v(t,x),\;\;\;  t>0,
\\
v(0,x) = \varphi (x),\;\; x \in \R^d,
\end{cases}
$$ 
where ${\cal L}_{n} \phi (x) =  \tfrac{1}{2} Tr (a_{n}(x) D^2 \phi (x)) + b_{n}(x) \cdot D\phi(x)$. Note that 
$$
\begin{cases}
\partial_t u(t,x) = {\cal L}_n u(t,x) - {V_{n}}(x) u(t,x) + {V_{n}}(x)u(t,x),\;\;\;  t>0,
\\
u(0,x) = \varphi (x),\;\; x \in \R^d,
\end{cases}
$$  
has the unique bounded classical solution $u(t,x) = P_t^{n} \varphi (x)$.  It follows that
\begin{equation}
\label{a21}
P_t^n\varphi(x)=S_t^{n}\varphi(x)+\int_0^t S_{t-s}^n({V_{n}} \psi_s^n )(x)\,ds,
\end{equation}
where $\psi_s^n(x) = P_s^n\varphi(x)$ is a regular bounded function.
 We deduce a  formula for the directional derivative of $P_t^n \varphi$ along the direction $h$:
\begin{equation}
\label{a211}
 \langle D P_t^n\varphi(x), h \rangle
= D_h P_t^n\varphi(x)= D_h S_t^n\varphi(x)+\int_0^t D_h S_{t-s}^n({V_{n}} \psi_s^n )(x)\,ds,
\end{equation}
 $ t>0,\, x,h \in \R^d.$
We want to obtain an estimate for   $D_h P_t^n\varphi $ as
\begin{equation}
\label{st}
 |D_hP_t^n\varphi(x)| \le C |h|  \| \varphi\|_{\infty} 
\tfrac{1}{\sqrt {\nu \,t}} V^{2}(x), \;  x,h \in \R^d,\, t\in (0,1]. 
 \end{equation}
From \eqref{st} assertion \eqref{dop1a} follows easily.
 We  have already proved an  estimate for $D_h S_t^n\varphi$ in Lemma \ref{l1.3} (note that there exists $C>0$ such that $V(x) \le CV^2(x)$, $x \in \R^d$). 

As before we write
$
X_{n}(t,x) = X_{n}(t),$ $\eta^h_{n}(t,x) = \eta_{n}(t)
$, $V' = DV$ and consider the process ${\beta_{n}}(t)$ given in \eqref{be1}. 

We fix $t \in (0,1]$, $s \in [0,t)$;  setting 
$
r = t-s
$
we find
\begin{gather*}
\langle D S_{r}^n({V_{n}} \psi_s ) (x), h \rangle
\\
 =\tfrac{1}{r}\,\E \left[  {V_{n}}(X_n(r)) \psi_s^n (X_n(r))\,e^{-\beta_n(r)}\,\int_0^{r}  \langle \sigma^{-1}_{n}(X_n(p)) \, \eta_n(p),dW(p)   \rangle \right]
\\
\ds-\E \left[  {V_{n}}(X(r,x)) \psi_s^n  (X_n(r))\,e^{-\beta_n (r)}\int_0^{r}\left(1-\tfrac{p}{r}   \right)\,\langle    {V_{n}}'(X_n(p)),\eta_n(p)  \rangle \,dp \right]
\\
=: I_1( \psi_s^n ,x,h,r,n)+I_2( \psi_s^n ,x,h,r,n)=I_1+I_2,
\end{gather*}
 for $  x,h \in \R^d$
(cf. \eqref{e1.11}). Since $ \| \psi_s^n \|_{\infty} \le \| \varphi\|_{\infty}$, $s \ge 0,$ $n \ge 1,$ we get arguing as in the proof of Lemma  \ref{l1.3} (we are using  similar notations)
\begin{equation}
\label{e14e1}
\begin{array}{l}
\ds |I_1|^2 \le  \| \varphi\|_{\infty}^2 \ds\tfrac1{r^2}\left[\E\left( V_{n}^2(X_n(r)\right)\right] \,  \E\left(e^{-2 \beta_n(r)}\left|\int_0^r \langle   \sigma^{-1}_{n}(X_n(p)) \eta_{n}(p),dW(p)\rangle\right|^2    \right)  \\
\\
\hspace{6mm}
\le \ds \tfrac2r \, \tfrac{ C |h|^2}{\nu} \, \| \varphi\|_{\infty}^2 \, \E\left[ V_{n}^2(X_n(r)) \right].
\end{array} 
\end{equation}
Now we use \eqref{ww1} and obtain
  \begin{equation}
\label{ee1}
|I_1|^2 
\le    \ds \tfrac2r \, \tfrac{  c \, |h|^2}{\nu} \, \| \varphi\|_{\infty}^2  \, V^2\left( x  \right).
\end{equation}
Now we treat  $I_2 $, 
$$
I_2 = -\E \left[ {V_{n}}(X_n(r)) \psi_s^n(X_n(r))\,e^{-\beta_n(r)}\int_0^r\left(1-\tfrac{p}{r}   \right)\,\langle    V_{n}'(X_n(p)),\eta_{n}(p)  \rangle \,dp \right ].
$$
We argue as in the estimates before \eqref{e29}:
\begin{equation}
\label{e291}
|I_2|
\le C \, \| \varphi\|_{\infty} \,   \Gamma_n(r),
\end{equation}
\begin{equation}
\label{e30}
\begin{array}{l}
\ds\Gamma_n(r)= \Gamma_n(r,x)=\E\left[ {V_{n}}(X_n(r)) \,  \int_0^r {V_{n}}^{}(X_n(p)) e^{-\beta_n(p)} \,|\eta_{n}(p)|\,dp \right].
 \end{array}
\end{equation} 
 Using the generalized H\"older inequality, since $\tfrac14 + \tfrac14 +\tfrac12 =1$,  estimate \eqref{ww1} and Lemma \ref{l1.2} we can compute 
\begin{gather*}
\E\big[  V_{n}(X_n(r)) \,  V_{n}(X_n(p))^{} e^{-\beta_n(p)} \,|\eta_{n}(p)| \big]
\\
\le (\E\big[  V_{n}^4(X_n(r))])^{1/4}  (\E\big[   V_{n}^4(X_n(p)))^{1/4}
( \E\big[  e^{-2\beta_n(p)} \,|\eta_{n}(p)|^2 \big])^{1/2}
\\
\le C |h|   \, V^{2}(x), \;\; p \in [0,r].
\end{gather*}
It follows that
$$
 \Gamma_n(r) \le  C |h|   \, V^{2}(x).
$$
Collecting the previous estimates we obtain
$$
| \langle D S_{t-s}(V_{n} \psi_s^n ) (x), h \rangle |
\le C |h|   \| \varphi\|_{\infty} \Big( \tfrac{1}{\sqrt{\nu (t-s)}} + 1 \Big) V^{2}(x).
$$
We finally get,  for $t \in [0,1],$ $x \in \R^d,$ 
$$
\int_0^t | \langle D S_{t-s}^n({V_n} \psi_s^n ) (x), h \rangle | ds
\le  \tfrac{c}{\sqrt{\nu}} |h|   \| \varphi\|_{\infty} 
V^{2}(x).
$$
for some $c = c(L, d,M_0, \gamma, c_0)>0$.
 Now \eqref{st} follows.   Hence we have proved 
 \begin{equation} \label{322}
 |P_t^n \varphi(x + h) - P_t^n \varphi(x)| \le C |h|  \| \varphi\|_{\infty} 
\tfrac{1}{\sqrt {\nu \,t}} {V}^{2}(x), \;  x,h \in \R^d,\, t\in (0,1].
 \end{equation}
 The proof is complete.

\section{Proof of Theorem \ref{mai} }

{\it I Step. } Let us fix $\varphi \in C_b(\R^d)$. We first prove 
the Lipschitz estimate
\begin{equation} \label{322a}
 |P_t \varphi(x + h) - P_t \varphi(x)| = |\E \big [\varphi (X(t, x +h)) -\varphi (X(t, x )) \big] |
\end{equation}
$$
\le C |h|  \| \varphi\|_{\infty} 
\tfrac{1}{\sqrt {\nu \,t}} {V}^{2}(x), \;  x,h \in \R^d,\, t\in (0,1].
$$
We deduce such estimate by a localization argument, passing to the limit in  \eqref{322}. To this purpose, let $x , h \in \R^d$ and set $X(t) = X(t, x+h),\;\; Y(t) = X(t,x)$,
$$
\;\; X_n(t) = X_n(t, x+h),\;\; Y_n(t) = X_n(t,x).
$$
Note that the $\R^{2d}$-valued processes  $Z(t) = (X(t), Y(t))$ and 
$Z_n(t) = (X_n(t), Y_n(t))$ are the unique strong solutions of  the SDEs
$$ 
\left\{\begin{array}{l}
dZ(t)= \tilde b(Z(t))dt+  \tilde \sigma (Z(t))dW(t),\\
Z(0)=(x+h,x)\in \R^{2d},
\end{array}\right. 
$$
$$ 
\left\{\begin{array}{l}
dZ_n(t)= \tilde b_n(Z_n(t))dt+  \tilde \sigma_n (Z_n(t))dW(t),\\
Z_n(0)=(x+h,x)\in \R^{2d},
\end{array}\right. 
$$
where 
$$
\tilde b (x,y) = \tilde b(z) =  \begin{pmatrix}
 b(x)\\
 b(y)
\end{pmatrix},\;\;\; \tilde \sigma (z) = \tilde \sigma (x,y) =  \begin{pmatrix}
 \sigma(x)\\
 \sigma(y)
\end{pmatrix},
$$
$$
\tilde b_n (x,y) = \tilde b_n(z) =  \begin{pmatrix}
 b_n(x)\\
 b_n(y)
\end{pmatrix},\;\;\; \tilde \sigma_n (z) = \tilde \sigma_n (x,y) =  \begin{pmatrix}
 \sigma_n(x)\\
 \sigma_n(y)
\end{pmatrix}
$$
($b_n$ and $\sigma_n$ are the same as in Section 3).
Let us consider the open set $A_n = B(0,n) \times B(0,n) \subset \R^{2d}$, $n \ge 1,$ and consider the stopping times 
$$
 \tau_n^Z = 
 \tau_n^Z(x+h,x)  = \inf \{t \ge 0 \; :\; Z(t) \not \in A_n  \} = \tau^X_n \wedge \tau^Y_n,
$$
where $\tau^X_n = \tau^X_n (x+h) = \inf \{t \ge 0 \; :\; |X(t)| \ge n \}$; 
 $\tau^Y_n = \tau^Y_n (x) = \inf \{t \ge 0 \; :\; |Y(t)| \ge n \}$. Moreover
$$
\tau_n^{Z_n} = 
 \tau_n^{Z_n}(x+h,x)  = \inf \{t \ge 0 \; :\; Z_n(t) \not \in A_n \}.
$$
Since we know that 
$$
\tilde b_n (z) =  \tilde b (z),\;\; \tilde \sigma_n (z) =  \tilde \sigma (z),\;\; z \in A_n,\, n \ge 1,
$$
and all the coefficients of the SDEs are locally Lipschitz, by a well-known localization principle  we obtain that, for any $n \ge 1$, $\P$-a.s., 
\begin{equation}\label{fhh} 
 \tau_n^Z = \tau_n^{Z_n},\;\;\text{and}\;\; Z(t \wedge \tau_n^Z ) = Z_n(t \wedge \tau_n^{Z_n} ),\; t \ge 0.
\end{equation}
Let us fix $t >0.$
 We   show now that $Z_n(t)$ converges in law to $Z(t)$ (this fact could be also deduced by a general convergence result given in Theorem V.5 of \cite{Kr95};  we provide a direct proof).
 Let $F(x,y)$ be a real continuous and bounded function on $\R^{2d}$. We have
\begin{gather*} 
|\E [F(Z(t)) - F(Z_n(t)) ]| = |\E [F(Z(t)) - F(Z_n(t)) \,  1_{ \{  t \le \tau_n^Z \}}]| 
\\
+ |\E [F(Z(t)) - F(Z_n(t)) \,  1_{ \{  t > \tau_n^Z \}} ]|
\le 2 \|F  \|_{\infty} \, \P ( t > \tau_n^Z).
\end{gather*}
We have that  $\P ( t > \tau_n^Z) \le \P ( t > \tau_n^X) + \P ( t > \tau_n^Y)$ and we know that 
$$
\lim_{n \to \infty} (\P ( t > \tau_n^X) + \P ( t > \tau_n^Y)) =0
$$
thanks to hypothesis (H1) which implies the non-explosion condition \eqref{111}.
Hence $|\E [F(Z(t)) - F(Z_n(t)) ]| $ tends to $0$ as $n \to \infty$ and this shows the desired weak convergence.

Writing   \eqref{322} with  $F(u,v) = \varphi(u) - \varphi(v)$, $u,v \in \R^d,$ we have:
\begin{gather*} 
|\E \big [F (Z_n(t)) \big] | = |\E \big [\varphi (X_n(t, x +h)) -\varphi (X_n(t, x )) \big] |
\\
 \le   C |h|  \| \varphi\|_{\infty} 
\tfrac{1}{\sqrt {\nu \,t}} {V}^{2}(x).
\end{gather*}
By weak convergence, passing 
  to the limit as $n \to \infty$ and get \eqref{322a}.

\smallskip \noindent
{\it II Step. }   By Section 4 in \cite{MPW} we deduce  in particular that 
$P_t \varphi \in C^2(\R^d)$, $t>0$. To see this fact, let us fix $\varphi \in C_b(\R^d)$; recall that 
\begin{equation}\label{exp1} 
P_t \varphi(x) = \lim_{n \to \infty} \E [\varphi(X(t,x)) \, 1_{\{  \tau_{n}^x >t \}}], \;\;\; t>0, \; x \in \R^d,
\end{equation}
where, for $n\ge 1$,  $\tau_n^x =  \inf \{t \ge 0 \; :\; |X(t,x)| \ge n\}.$ The previous formula holds because  by (H1)  we know that, $\P$-a.s., $\tau_n^x \to \infty$, as $n \to \infty$.
 
As in \cite{MPW} we can apply the classical interior Schauder estimates by A. Friedman to each $u_n(t,x) = \E [\varphi(X(t,x)) \, 1_{\{  \tau_{n}^x >t \}}]$; by  Theorem 4.2 in \cite{MPW} 
 we obtain that $P_t \varphi \in C^{1 + \alpha/2, 2 + \alpha}_{loc} ((0, \infty) \times \R^d)$, for any $\alpha \in (0,1)$.
 
Using that  $P_t \varphi \in C^2(\R^d)$, $t>0$, and our estimate \eqref{322a}
we easily obtain the gradient estimate \eqref{dop1}.
   \qed

\begin{Remark}{\em (i) By a well-known argument, see Lemma 7.1.5 in \cite{DaZa96}, gradient estimates \eqref{dop1} hold also when $\varphi \in B_b(\R^d)$, i.e., we have
\begin {equation} \label{dop}
{\sqrt{t\nu}}  \, |D_x P_t \varphi (x) | \le 
c \, \, V^{2}(x)\| \varphi\|_{\infty}, \;\;\;  \text{$t \in (0,1]$, $\varphi \in B_b (\R^d)$, $x \in \R^d.$}
\end{equation}
(ii) Using the semigroup law, it is easy to check that, for any $t>1$, we have  (with $c = c(M_0, \gamma, L, c_0,d)>0$)
$$
 \,\sqrt{\nu} |D_x P_t \varphi (x) | \le 
c \, \, V^{2}(x)\| \varphi\|_{\infty}, \;\;\;  \text{ $\varphi \in B_b (\R^d)$, $x \in \R^d.$}
$$
}
  \end{Remark}

\section{A counterexample to uniform gradient estimates }

We mention that a  one-dimensional counterexample to uniform gradient estimates is given in \cite{BeFo04}. 
 This  concerns with  diffusion  semigroups   
 having an invariant measure. In  such example
  the drift term $ b(x)$  grows faster than $e^{x^4}$ for $x \to +\infty$ and the  non-explosion is guaranteed by the existence  of an   invariant measure.



\smallskip
We consider transition Markov semigroups   $(P_t)$ associated to 
 the one-dimensional SDEs 
$$
d X(t) = b(X(t)) dt +    \sqrt{2}\, dW(t),\;\;\; X(0) =x \in \R,
$$
with $b \in C^1(\R)$. 
We  assume that there exists $\theta \in (0,1)$
 and $k_0>0$ 
 such that 
\begin{gather} 
\nonumber |b(x)| \le k_0(1 + |x|^{\theta});
\\
|b'(x)| \le k_0(1 + |x|^{2\theta}),\;\;\; x \in \R.
\label{d2}
\end{gather}
 Before stating the next theorem we recall that  if $b \in C^1(\R)$ is also bounded then uniform gradient estimates hold. This is a special case of a  result given in  \cite{St}.

\begin{Theorem}  Let us fix $\theta \in (0,1)$.
There exists $b \in C^1(\R)$ which satisfies \eqref{d2} for some $k_0$ 
such that,  for any $c>0,$
the inequality    
\begin{gather} \label{g1}
\| DP_t f\|_{\infty} \le \tfrac{c}{ \sqrt{t}} \| f \|_{\infty} \;\; \text{for any} \; t \in (0,1], \;\; \text{for any} \; f \in C_b(\R),
\end{gather}
\textit{does not hold}, i.e., 
$ \displaystyle{
\sup_{t \in (0,1]} \; \sup_{f \in C_b(\R),\, \| f\|_{\infty} \le 1} \big(\sqrt{t}\, \| DP_t f\|_{\infty} \big)  = \infty.}
$
\end{Theorem}
 Note that a drift  $b$ which   satisfies \eqref{d2} also verifies  Hypothesis \ref{h1} with $\gamma =1$ and $V(x) = k_0 (1+ |x|^2)$
  so that we can apply Theorem  \ref{mai} to $(P_t)$. 

  \subsection{Preliminaries}

The SDE is  associated to 
$$
L u(x) = 
 u''(x) + b(x) u'(x),\;\;\; x \in \R. 
$$
 Under our assumptions $b$ grows at most linearly and so
$
\P (\tau_N > t ) \to 1 \;\; \text{as} \; N \to \infty.
$
where $\tau_N = \tau_N^x = \inf\{t \ge 0, \; :\; |X(t,x)|\ge N \}$, $N \ge 1$. Let now $g \in C_b(\R)$. 

If there exists  $u \in C^2 (\R) \cap C_b (\R) $ such that  
\begin{gather*} 
L u(x) =g(x),\;\; x \in \R. 
\end{gather*}
then,  by the It\^o formula we know that, for any $N \ge 1,$ 
$$
\E[ u(X ({t \wedge \tau_N^x },x))] = u (x)  + \E \Big [\int_0^{t \wedge  \tau_N^x } g (X(s,x)) ds \Big].
$$
Passing to the limit as $N \to \infty$, we obtain 
\begin{equation}\label{ma1} 
P_t u(x) = u(x) + \int_0^t P_s g(x)ds,\;\;\;\;\; t \in [ 0,1), \; x \in \R.
\end{equation}
Assume by contradiction that uniform gradient estimates 
\eqref{g1}  hold. 
Then we would  have 
$$
 \| D u \|_{\infty} \le 
\| DP_1 u \|_{\infty} +  c  \int_0^1 \tfrac{1} {\sqrt{s}} ds \; \| g\|_{\infty} \le C (\| u\|_{\infty} + \| g \|_{\infty}) < \infty.
$$
for some $C>0$. Hence,  uniform gradient 
 estimates  \eqref{g1} do not hold if we show 
\begin{Lemma} \label{sss} 
 Let us fix $\theta \in (0,1)$.
There exist $b \in C^1(\R)$ which satisfies \eqref{d2} for some $k_0 >0$, a function  $f \in C_b(\R) $ and a  
  bounded classical solution $u$ to $Lu =f$ with  unbounded derivative, i.e., 
 $\sup_{x  \in \R} |u'(x)| = \infty.$
\end{Lemma}
   For the proof we need two elementary results about alternating series and improper Riemann integrals. 
 The first one can be found in pages 637 and 638 of \cite{Joh}.

\begin{Proposition} \label{jo} 
Let $(a_k)_{k \ge 3}$ be a (strictly) decreasing sequence of positive numbers with limit 0 such that  $(a_k - a_{k+1})_{k \ge 3}$ is decreasing too.
Consider
$$
R_n = \sum_{k \ge n} (-1)^{k+1} a_k, \;\;\; n\ge 3. 
$$
 Then $R_n = (-1)^{n+1} |R_n|$ and 
\begin{equation} \label{q1}
\tfrac{a_n}{2} \le |R_n| \le \tfrac{a_{n-1}}{2},\;\; n \ge 4.
\end{equation} 
\end{Proposition}
Note that $a_k = \tfrac{1}{k^{\gamma}}$, $\gamma >0$, verifies the previous conditions. 

Under the assumptions of  the previous result,  by \eqref{q1} we deduce  that  $|R_n|$ is decreasing and has limit 0. Hence, by the Leibnitz criterion, 
 \begin{equation} \label{u22}
\sum_{n \ge 4} R_n \;\; \text{converges}.
\end{equation}
We will use \eqref{u22} in the sequel.

\begin{Proposition} \label{aa} Let $a \in \R$ and let $g : [a, \infty) \to \R$ be a continuous function. Consider an increasing sequence $(z_n)_{n \ge n_0} \subset  [a, \infty)$ such that $z_{n_0}=a$ and  $\lim_{n \to \infty} z_n = \infty$.
  Then there exists $I = \lim_{x \to \infty} \int_a^x g(t) dt$ if the following two conditions are satisfied:

 \smallskip 
(i) \  $\sum_{n \ge n_0} \int_{z_n}^{z_{n+1}} g(t) dt = I$;

(ii) \  $\lim_{n \to \infty} \sup_{x \in [z_n, z_{n+1}]} \big | \int_{z_n}^{{x}} g(t) dt\big| =0.$ 
\end{Proposition}
\begin{proof} Let $\epsilon >0$. There exists $N = N_{\epsilon} > n_0$ such that, for any $n \ge N,$ 
$$
 \Big | \sum_{k = n_0}^n \int_{z_k}^{z_{k+1}} g(t) dt - I \Big |
=  
\Big | \int_{a}^{z_{n+1}} g(t) dt - I \Big | < \epsilon
$$
and  $\sup_{x \in [z_n, z_{n+1}]} \big | \int_{z_n}^{{x}} g(t) dt\big| < \epsilon$. Now if  $x \ge z_{N+1}$ then $x \in [z_{N+k}, z_{N+k+1} ]$ for some $k \ge 1$ and we find
$$
\Big | \int_{a}^{x} g(t) dt - I \Big |
 \le \Big | \int_{a}^{z_{N+k}} g(t) dt - I \Big | +
 \sup_{x \in [z_{N+k}, z_{N+k+1}]}
 \Big | \int_{z_{N+k}}^x g(t) dt  \Big |
< 2\epsilon.
$$
 \end{proof}
\subsection {Proof of Lemma \ref{sss}
}

Solutions $u$ to $Lu =f$ are given by 
$$
u (x)= c_0 + \int_0^x e^{- B(t)} \Big (c_1 +  \int_0^{t} f(s) e^{B(s)} ds \Big)dt,\;\;\; x \in \R,
$$
where $B(t) = \int_0^t b(r) dr$.  We set $c_0 =0$.
We will construct $b$ satisfying \eqref{d2} and  a suitable changing  sign function $f \in C_b (\R)$ with the property 
that  
\begin{equation*}
\int_{- \infty}^{0} |f(t)| e^{B(t)}dt < \infty
\;\;\; \text{and}
 \end{equation*}
\begin{equation} \label{112}
J=  \int_0^{ \infty} f(t) e^{B(t)}dt 
\;\;\; \text{is a convergent  improper   integral}
 \end{equation}
(the second integral will be only conditionally convergent since $f(t) e^{B(t)}$ will be  not Lebesgue integrable on $\R_+$). Moreover we will require that 
\begin{equation} \label{segno}
 J = - c_1\;\;\text{where} \, c_1 =
\int_{- \infty}^{0} f(t) e^{B(t)}dt.
\end{equation}
Therefore
$$
\int_{-\infty}^{+\infty} e^{B(s)}\,f(s)ds=0.
$$
 Under the previous conditions we have:
\begin{gather*}
 u(x) = - \int_x^0 e^{- B(t)} \Big (c_1 -  \int_t^{0} f(s) e^{B(s)} ds \Big)dt
\\ =
 - \int_x^0 e^{- B(t)} \Big ( \int_{- \infty}^{t} f(s) e^{B(s)} ds \Big)dt
,\;\;\; x <0;
\end{gather*}
\begin{equation} \label{piu}
 u(x) = - \int_0^x e^{- B(t)} \int_t^{\infty} f(s) e^{B(s)} ds dt,\;\;\; x \ge 0,
\end{equation}
where $ \int_t^{\infty} f(s) e^{B(s)} ds  = \lim_{M \to \infty}  \int_t^{M} f(s) e^{B(s)} ds  $ exists as an improper  Riemann integral. Now we proceed in some steps.

\paragraph{I Step} {\sl To prove the lemma we need to 
 find $b \in C^1(\R_+)$ satisfying \eqref{d2} on $[0, \infty)$ and $f \in C_b(\R_+)$ such that condition \eqref{112} hold and moreover  
\begin{equation} \label{wwwa}
f(0)=0,\;\; b(0)=1,\;\; b'(0)=0.
\end{equation}
We also need that the $C^2$-function $u$ given in  \eqref{piu} verifies 
\begin{gather} 
 \sup_{x \in \R_+}|u(x)| < \infty \;\; \text{and}
\label{bo1}
\\
 \label{gra1}
 \limsup_{x \to \infty} |u'(x)| = \infty.
\end{gather}
}
Indeed once this is done, we can easily extend $b$ and $f$ to $(-\infty, 0]$ as follows: 
$$
b (x) =1, \;\;\; f(x) = k x e^{x},\;\; x<0,
$$
so that, for $x<0,$  
$$
u(x)= - k\int_x^0  e^{-t}\big ( \int_{- \infty}^{t} s e^{2s} ds \big)dt = \tfrac{k}{4} \int_0^x e^{-t} (2te^{2t} - e^{2t}) dt = \tfrac{k}{4} (2x e^{x} - 3e^{x} +3), 
$$
where  $k \in \R$ is such that 
$
J = - \int_{- \infty}^0 f(r) e^{ B(r)} dr 
= 
- k \int_{- \infty}^{0} t e^{2t} dt =k$ and so \eqref{segno} holds.

\paragraph{Step 2} {\sl We define $b$ on $[0, \infty)$.}

 We  first consider a suitable  positive  $C^2$-function $l :  [0, \infty) \to \R_+$ such that 
$$
l(x) =  e^{- \int_0^x {b(s)}  ds }= e^{-B(x)},\;\;\; x \ge 0;
$$
it follows that $ -\log (l(x)) = \int_0^x {b(s)}  ds $ and we have
\begin{equation} \label{dee}
b(x) =  - \tfrac{l'(x)}{l(x)}, \;\;\;\;  x \ge 0.
\end{equation}
The function $l$ is defined as follows.  

We start with $\phi \in C_0^{\infty} (\R)$ such that Supp$(\phi) \subset [-1,1]$, $0 \le \phi \le 1$, $\phi(0)=1$ and 
$
 \int_{\R} \phi (t) dt =1.
$

Let us introduce the following sequences of positive numbers $(c_n),$ $(\delta_n)$ and $(b_n)$
 \begin{equation}
\label{211}  
c_n = n, \;\;\; \delta_n = \tfrac{1}{n^{3 \gamma}}, \;\; b_n = n^{2\gamma} -1, \;\;
n \ge 3, 
\end{equation}
with  $\gamma = \tfrac{\theta}{5}\in (0,1)$ (see condition \eqref{d2}). The function $l$ is given by 
\begin{equation} \label{qqq}
\begin{cases}
l(x) = 1+ b_n \phi \big (\tfrac{x- c_n}{\delta_n} \big),\;\;\text{if } \; x \in [c_n - \delta_n , c_n + \delta_n], \;\text{for some } \; n \ge 3;
\\ \\
 l(x) =1,  \;\; \text{ if $x \not \in  \bigcup_{n \ge 3} [c_n - \delta_n , c_n + \delta_n]$.}
\end{cases}
 \end{equation}
We have $ l(c_n) = 1 + b_n = n^{2 \gamma}$, $l (c_n - \delta_n ) = l(c_n + \delta_n) =1$. In the sequel we will also use that   
\begin{equation} \label{2ee}
\int_{c_n - \delta_n}^{c_n + \delta_n} l(x) dx = 2 \delta_n + \delta_n b_n = \delta_n (n^{2\gamma} - 1 + 2) =  \delta_n + \tfrac{1}{n^{\gamma}}.
\end{equation}
Note that $b$ satisfies  the first estimate in  \eqref{d2} since $|b(x)| = \tfrac{|l'(x)|}{|l(x)|} =0$ if $x \not \in \bigcup_{n \ge 3} [c_n - \delta_n , c_n + \delta_n] $ and if $x \in [c_n - \delta_n , c_n + \delta_n] = [n - \tfrac{1}{n^{3 \gamma}}, n + \tfrac{1}{n^{3 \gamma}}]$ 
then with $C = C(\phi) >0$, since $l(x)|\ge 1$,
\begin{gather} \nonumber
|b(x)| = \tfrac{|l'(x)|}{|l(x)|} \le \| \phi'\|_{\infty} \tfrac{b_n}{ \delta_n} = \| \phi'\|_{\infty} \tfrac{ n^{2\gamma} -1}{ \delta_n} = \| \phi'\|_{\infty}
(n^{5 \gamma } - n^{3\gamma})\\ \le C n^{5 \gamma}
 \le C \big( n - \tfrac{1}{n^{3 \gamma}}  \big)^{5 \gamma} + 1 \le 
C( x^{5 \gamma} + 1) \le C( x^{\theta} + 1).
 \label{dff}
\end{gather}
This shows  the first estimate in  \eqref{d2} on $[0, \infty)$ (as required in Step 1).   To prove that 
 $$
|b'(x)| \le c (1 + |x|^{2\theta}), \;  x \in \R,
$$
with $c = c(\phi)>0,$ we argue as in \eqref{dff};  for $x \in [c_n - \delta_n , c_n + \delta_n] = [n - \tfrac{1}{n^{3 \gamma}}, n + \tfrac{1}{n^{3 \gamma}}]$, we have with $C' = C' (\phi) >0$: 
\begin{gather*}
|b'(x)| \le \tfrac{b_n}{\delta_n^2} \| \phi''\|_{\infty} + 
\tfrac{b_n^2}{\delta_n^2} \| \phi'\|_{\infty}
\\ \le 
C'  n^{10 \gamma } 
\le C' \big( n - \tfrac{1}{n^{3 \gamma}}  \big)^{10 \gamma} + 1 \le
C' (x^{10 \gamma} +1) \le C' (x^{2 \theta} +1). 
\end{gather*}
This 
completes the proof of 
 \eqref{d2}.

\paragraph{Step 3} {\sl We define $f$ on $[0, \infty)$. }

Let $a>0$. We  first consider 
the function $g(s)=1-\tfrac{|s|}{a},\;s\in[-a,a]$. We note that 
$$
 \int_{-a}^a  dt \int_t^a g(s) ds =  \int_{-a}^a  g(s) (s+a)  ds = a^2
$$
and similarly, for any $z \in \R$,
\begin{equation} \label{d33}
\int_{z-a}^{z+a}  dt \int_t^{z+ a} \Big (  1 - \tfrac{|s- z|}{a}\Big) ds =a^2.
\end{equation} 
Let us define, for $n \ge 3,$
$$
a_n = \tfrac{1}{n^{\gamma}},\;\;\; x_n = n + \tfrac{1}{2},
$$
where $\gamma = \tfrac{\theta}{5}$ as in \eqref{211}. 
 There exists an  odd number $n_0 = n_0(\gamma)\ge 3$ such that 
\begin{equation} \label{lista}
c_n - \delta_n = n - \tfrac{1}{n^{3\gamma}} <  c_n + \delta_n
< x_n - a_n = (n+ \tfrac{1}{2}) - \tfrac{1}{n^{\gamma}} < x_n + a_n < c_{n+1} - \delta_{n+1}.
\end{equation}
To this purpose it is enough to choose $n_0$ such that
 $\tfrac{1}{n^{3\gamma}_0} +
 \tfrac{1}{n^{\gamma}_0} < 1/2.$
 The function $f: [0, \infty) \to \R$ is defined as follows.  
\begin{equation} \label{cdd}
\begin{cases}
f(x) =  (-1)^{n+1} \Big (  1 - \tfrac{|x- x_n|}{a_n}\Big),\;\;\; \text{if} \; x \in [x_n - a_n , x_n + a_n],\; \text{for some} \;n \ge n_0; 
\\ \\
f(x)=0 \;\;\; \text{ if $x \not \in  \bigcup_{n \ge n_0} [x_n - a_n , x_n + a_n]$.}
\end{cases}
 \end{equation}
Note that, according to \eqref{qqq},  since
 $l$ is identically 1 on the support of $f$, 
\begin{equation} \label{wss}
\tfrac{f(x)}{ l(x)} = f(x),\;\;\; x \ge 0.
\end{equation}
Moreover,
$$
\lim_{n \to \infty} \sup_{x \in [x_n - a_n, x_n + a_{n}] } \Big |\int_{x_n - a_n}^{x}   {f(s)}  ds \Big | = 0.
$$
 Recall that  $\displaystyle{\int_{x_k-a_k}^{x_k + a_k }  f(t) dt = (-1)^{k+1} a_k.}$
By Proposition \ref{aa} it follows that
\begin{equation} \label{ffr}
\int_{x_n-a_n}^{\infty}  f(t) dt
= \sum_{k \ge n}\int_{x_k-a_k}^{x_k + a_k }  f(t) dt
= 
\sum_{k \ge n} (-1)^{k+1} a_k,\;\;\; n \ge n_0. 
\end{equation}
Hence, in particular,
$$
\int_{0}^{\infty} \tfrac{f(x)}{ l(x)} dx =
\int_{x_{n_0} -a_{n_0}}^{\infty}  f(t) dt
= 
\sum_{k \ge n_0} (-1)^{k+1} a_k
$$
is a convergent integral and  \eqref{112} holds.
 In the sequel we will also use that 
\begin{equation} \label{www}
 \int_{x_n-a_n}^{x_n+a_n}  dt \int_t^{x_n+ a_n} f(s)ds
 =(-1)^{n+1}\,  a^2_n, \;\; n \ge n_0.
\end{equation}
\paragraph {Step 4} {\sl We check that  $u$ given in  \eqref{piu}
verifies \eqref{bo1} and \eqref{gra1}.}

Note that, 
 for $ x \ge c_{n_0 - \delta_{n_0}},$
 \begin{equation} \label{piu1}
 u(x) = - \int_0^x l(t) dt \int_t^{\infty} \tfrac{f(s)}{l(s)}  ds = 
 C_0  
 - \int_{c_{n_0 - \delta_{n_0}}}^x l(t) dt \int_t^{\infty} {f(s)}  ds,
\end{equation}
where $C_0 =- \int_0^{c_{n_0 - \delta_{n_0}}} l(t) dt \int_t^{\infty} \tfrac{f(s)}{l(s)}  ds  $ thanks to \eqref{wss}. 

Let us first check \eqref{gra1}. We have, for $n \ge n_0$,
taking into account \eqref{ffr} and \eqref{q1},
\begin{gather*}
  |u'(c_n)| =  \Big | l(c_n) \int_{c_n}^{\infty} f(s)  ds  \Big| = (b_n +1) \Big | \int_{x_n -a_n}^{\infty} f(s)  ds  \Big|\\  = n^{2 \gamma}   \Big | \sum_{k \ge n} (-1)^{k+1} a_k
\Big|  
\ge n^{2 \gamma}  \, \tfrac{a_{n}}{2} = n^{2 \gamma}  \, \tfrac{1}{2 n^{\gamma} } \to \infty,
\end{gather*}
as $n \to \infty$. This shows \eqref{gra1}.
 The proof of \eqref{bo1} is more involved. By \eqref{piu1} it is enough to verify that there exists
\begin{equation} \label{imp1}
\lim_{x \to \infty} \int_{c_{n_0 - \delta_{n_0}}}^x l(t) dt \int_t^{\infty} {f(s)}  ds = I \in \R.
\end{equation}
\paragraph {Step 5} {\sl We  check \eqref{imp1} by Proposition \ref{aa}.}

We first 
prove that
\begin{equation} \label{se11}
 \sum_{n \ge n_0}  \int_{c_n - \delta_n}^{c_{n+1} - \delta_{n+1}}  l(t) dt \int_t^{\infty} {f(s)}  ds \;\; \text{is convergent.}
\end{equation}
We write, using \eqref{lista},
\begin{gather} \label{rit11}
  \int_{c_n - \delta_n}^{c_{n+1} - \delta_{n+1}}  l(t) dt \int_t^{\infty} {f(s)}  ds 
  \\
\nonumber   =  \int_{c_n - \delta_n}^{c_{n} + \delta_{n}}  l(t) dt \int_t^{\infty} {f(s)}  ds 
  +  
   \int_{c_n + \delta_n}^{x_{n} - a_{n}}  l(t) dt \int_t^{\infty} {f(s)}  ds 
\\   \nonumber 
 +  
    \int_{x_n - a_n}^{x_{n} + a_{n}}  l(t) dt \int_t^{\infty} {f(s)}  ds     
 +    \int_{x_n + a_n}^{c_{n+1} - \delta_{n+1}}  l(t) dt \int_t^{\infty} {f(s)}  ds 
\\ \nonumber
= A_n + B_n + C_n + D_n.
\end{gather}
To verify  \eqref{se11} it is enough to prove that $\sum_{n \ge n_0} (A_n +B_n)$ and 
$\sum_{n \ge n_0} (C_n +D_n)$ 
are convergent. 

We have by \eqref{ffr} and \eqref{2ee}, since $l =1$ on $[c_n + \delta_n, x_n - a_n]$,
\begin{gather*}
A_n + B_n =   \int_{c_n - \delta_n}^{c_{n} + \delta_{n}}  l(t) dt \int_{x_n - a_n}^{\infty} {f(s)}  ds
+ 
 \int_{c_n + \delta_n}^{x_{n} - a_{n}}  l(t) dt \int_{x_n - a_n}^{\infty} {f(s)}  ds
\\
=  \Big [ \int_{c_n - \delta_n}^{c_{n} + \delta_{n}}  l(t) dt +  \int_{c_n + \delta_n}^{x_{n} - a_{n}}   dt 
\Big] \sum_{k \ge n} (-1)^{k+1} a_k
\\
= 
  \Big ( \delta_n + \tfrac{1}{n^{\gamma}}  + x_n - a_n - c_n - \delta_n \Big) \sum_{k \ge n} (-1)^{k+1} a_k
\\
= \big ( x_n - c_n \big)  \sum_{k \ge n} (-1)^{k+1} a_k 
= \tfrac{1}{2} \sum_{k \ge n} (-1)^{k+1} \tfrac{1}{k^{\gamma}}.
\end{gather*}
By  Proposition \ref{jo} and the Leibnitz criterion (see \eqref{u22}) we infer that
\begin{gather*}
 \sum_{n \ge n_0} (A_n +B_n) = \tfrac{1}{2}
 \sum_{n \ge n_0}  \sum_{k \ge n} (-1)^{k+1} \tfrac{1}{k^{\gamma}} \;\; \text{is convergent.}
\end{gather*}
Now, since $l=1$   on $[ x_{n} - a_{n},c_{n+1} - \delta_{n+1} ]$, recalling the definition of $f$ we find
\begin{gather*}
 C_n + D_n =   \int_{x_{n} - a_{n}}^{x_{n} + a_{n}}  dt \int_{t}^{\infty} {f(s)}  ds
+ 
 \int_{x_n + a_n}^{c_{n+1} - \delta_{n+1}}   dt \int_{t}^{\infty} {f(s)}  ds
 \\
 = \int_{x_{n} - a_{n}}^{x_{n} + a_{n}}  dt \int_{t}^{x_n + a_n} {f(s)}  ds
+ 
\int_{x_{n} - a_{n}}^{x_{n} + a_{n}}  dt \int_{x_{n+1} - a_{n+1}}^{\infty} {f(s)}  ds
\\ + 
 \int_{x_n + a_n}^{c_{n+1} - \delta_{n+1}}   dt \int_{x_{n+1} - a_{n+1}}^{\infty} {f(s)}  ds.
\end{gather*}
Using \eqref{ffr}  and \eqref{www}  we get
\begin{gather*}
  C_n + D_n =  (-1)^{n+1}\,  a^2_n
 \\ + (2 a_n + c_{n+1} - \delta_{n+1} - x_n - a_n)  \sum_{k \ge n+1} (-1)^{k+1} \tfrac{1}{k^{\gamma}}  
\\ 
= (-1)^{n+1}\,  a^2_n + \tfrac{1}{2}\sum_{k \ge n+1} (-1)^{k+1} \tfrac{1}{k^{\gamma}} + (a_n   - \delta_{n+1})  \sum_{k \ge n+1} (-1)^{k+1} \tfrac{1}{k^{\gamma}}. 
  \end{gather*}
Applying again Proposition \ref{jo} and the Leibnitz criterion we obtain the convergence of $\sum_{n \ge n_0} C_n + D_n $ if we are able to show that the sequence
$$
(a_n - \delta_{n+1}) = \big (\tfrac{1}{n^{\gamma}} - \tfrac{1}{(n+1)^{3 \gamma}}\big ) \;\; \text{is definitively decreasing.}
$$
This can be easily checked by looking at the  function $g(s) = 
 \big (\tfrac{1}{s^{\gamma}} - \tfrac{1}{(s+1)^{3 \gamma}}\big )$; one proves that  there exists $n_1 = n_1(\gamma) > n_0$ such that 
 $(a_n - \delta_{n+1})_{n \ge n_1}$ is decreasing.
This shows  \eqref{se11}. Now let us define 
$$
I =  \sum_{n \ge n_0}  \int_{c_n - \delta_n}^{c_{n+1} - \delta_{n+1}}  l(t) dt \int_t^{\infty} {f(s)}  ds.
$$
According to Proposition \ref{aa} to obtain \eqref{imp1}, it remains to prove that 
\begin{equation} \label{fi1}
\Gamma_n = \sup_{x \in [c_n - \delta_n, c_{n+1} - \delta_{n+1}] } \Big |\int_{c_n - \delta_n}^{x}  l(t) dt \int_t^{\infty} {f(s)}  ds \Big | \to 0
\end{equation}
as $n \to \infty.$  Using the same notations of \eqref{rit11} we find
\begin{gather*}
 \Gamma_n \le 
2\sup_{x \in [c_n - \delta_n, c_n + \delta_n ]} 
 \int_{c_n - \delta_n}^{x}  l(t) \Big |  \int_{x_n -a_n}^{\infty} {f(s)}  ds \Big |dt 
 \\ + 2 
 \sup_{x \in [c_n + \delta_n, x_n - a_n]} 
 \int_{c_n + \delta_n}^{x}   \Big |  \int_{x_n -a_n}^{\infty} {f(s)}  ds \Big |dt 
\\
+  2\sup_{x \in [x_n - a_n, x_n + a_n ]} \Big ( \Big |
 \int_{x_n - a_n}^{x}   dt   \int_t^{x_n + a_n} {f(s)}  ds  \Big | 
\\ +   
 \int_{x_n - a_n}^{x}    \Big |  \int_{x_{n+1} - a_{n+1} }^{\infty} {f(s)}  ds  \Big | dt
  \Big)
 \\ + 2
 \sup_{x \in [x_n + a_n , c_{n+1} - \delta_{n+1}] } 
 \int_{x_n + a_n}^{x}   \Big |  \int_{x_{n+1} - a_{n+1}}^{\infty} {f(s)}  ds \Big |dt.
\end{gather*}
 By the previous computations involving  $A_n$, $B_n$, 
 $C_n$ and $D_n$, 
 using also  that
$$
 \Big |  \int_{x_{n} - a_{n}}^{\infty} {f(s)}  ds \Big | = 
\Big | 
\sum_{k \ge n} (-1)^{k+1} a_k 
\Big| \le \tfrac{a_{n-1}}{2} = \tfrac{1}{ 2 (n-1)^{\gamma}}, \;\; n \ge n_0
$$
(see \eqref{ffr} and Proposition \ref{jo}) it is straightforward to prove \eqref{fi1}. 
 This shows that \eqref{imp1} holds and finishes the proof.

\begin{Remark}  {\em  Concerning the drift term $b$ of the previous result, it can also be checked that the derivative $b'(x)$ is unbounded  from above  and from below. Thus condition \eqref{de22} does not hold for $b$.
}
\end{Remark}

\def\ciao{
\begin{Remark} {\em Note that the function $f$ of the previous lemma verifies 
$$
f \in L^{p_0}(\R)
$$
for $p_0$ such that $\gamma p_0 = \tfrac{\theta}{5}p >1$. By the Sobolev embedding Lemma \ref{sss} also provides a counterexample to Sobolev  elliptic regularity. Indeed with $b$ and $f$ as in the lemma, there exists a classical bounded solution $u $ to 
$$
u''(x) + b(x) u'(x) = f(x),\;\; x \in \R
$$
such that 
} 
\end{Remark}
}

  \subsection*{Acnowledgement}   The  authors were partially supported 
by the Istituto Nazionale di Alta Matematica (INdAM), through the GNAMPA
Research Project (2016)  ``Controllo, regolarit\`a e viabilit\`a per alcuni tipi di equazioni diffusive''.  The authors also thank  G. Metafune (Lecce) for useful discussions.

\end{document}